\documentclass{amsart}
\usepackage{amsfonts}
\usepackage{bbm}
\usepackage{mathrsfs}
\usepackage{amsfonts}
\usepackage{amsmath}
\usepackage{bigdelim}
\usepackage{multirow}
\usepackage{amssymb}
\usepackage{indentfirst}
\usepackage[dvips]{graphicx}
\usepackage{CJK}
\usepackage{fancyhdr}
\usepackage{amscd,amssymb,amsmath,graphicx,verbatim}
\usepackage[TS1,OT1,T1]{fontenc}

\newtheorem{theorem}{Theorem}[section]
\newtheorem{lemma}[theorem]{Lemma}

\theoremstyle{definition}
\newtheorem{definition}[theorem]{Definition}

\newtheorem{question}[theorem]{Question}
\theoremstyle{remark}

\numberwithin{equation}{section}

\begin{document}
\title{Li-Yorke chaos for invertible mappings on compact metric spaces}
\author{Lvlin Luo}
\address{School of Mathematics and Statistics, Xidian University, 710071, Xi'an, P. R. China.}\email{luoll12@mails.jlu.edu.cn}
\author{Bingzhe Hou$^*$ }
\address{Bingzhe Hou, College of Mathematics , Jilin university, 130012, Changchun, P.R.China} \email{houbz@jlu.edu.cn}
\thanks{Corresponding author}
\begin{abstract}
In this paper, we construct a homeomorphism on the unit closed disk to show that an invertible mapping on a compact metric space is Li-Yorke chaotic does not imply its inverse being Li-Yorke chaotic.
\end{abstract}

\maketitle

\section{Introduction}

A dynamical system $(X, f)$ means a compact metric space $X$ and a continuous map $f:X\rightarrow X$. Moreover, if $f$ is invertible, we call $(X, f)$ an invertible dynamical system; if $X$ is noncompact, we call $(X, f)$ a noncompact dynamical system. Since the 1980s there is a growing literature that looks
at the connection between inverse limits and topological dynamics. In particular, inverse limits dynamical systems are useful in economic theory such as the overlapping generations
model and the cash-in-advance model \cite{M06,M07,KennedyJRStockmanDAYorkeJ,RStockmanD}. Then the dynamical properties of inverse limit dynamical systems attract more attentions. For instance, L. Liu and S. Zhao showed that inverse limit dynamical system is Martelli's chaos if and only if so is the original system\cite{LiuLZhaoS}; X. Wu, X. Wang and G. Chen \cite{WuXWangXChenG} investigated some
chaotic properties via Furstenberg families generated by inverse limit dynamical systems. Notice that if  $f$ is invertible, there is no difference between the induced inverse limit dynamical system and $(X, f^{-1})$. So
we are interested in the following question.
\begin{question}
Let $(X,f)$ be an invertible dynamical system. If $f$ has a dynamical property $\mathfrak{P}$, does its inverse $f^{-1}$ also have the property $\mathfrak{P}$?
\end{question}
It is not difficult to see that the answer is positive for many properties such as
transitivity, mixing, Martelli's chaos, Devaney chaos, positive entropy and so on. However, it has been unknown for Li-Yorke chaos introduced by Li and Yorke in 1975 \cite{LY}, this open question is raised by D. Stockman \cite{RStockmanD}.
For noncompact invertible dynamical systems, we had given  negative answers for distributional chaos \cite{LH} and Li-Yorke chaos \cite{HouLuo}. Inspired by an example given in \cite{HouLuo}, we will construct a homeomorphism on the unit closed disk to show that an invertible mapping on a compact metric space is Li-Yorke chaotic does not imply its inverse being Li-Yorke chaotic.

\begin{definition}
Let $(X,f)$ be a dynamical system. $\{x,y\}\subseteq X$ is said to be a Li-Yorke chaotic pair, if
\begin{equation*}
\limsup\limits_{n\rightarrow+\infty}d(f^{n}(x),f^{n}(y))>0 \ \ and \ \
\liminf\limits_{n\rightarrow+\infty}d(f^{n}(x),f^{n}(y))=0.
\end{equation*}
Furthermore, $f$ is called Li-Yorke chaotic, if there exists an
uncountable subset $\Gamma\subseteq X$ such that each pair of two
distinct points in $\Gamma$ is a Li-Yorke chaotic pair.
\end{definition}

\section{Main result}

Denote $\mathbb{R}$ by the set of all real numbers, and denote $\mathbb{Q}$ by the set of all rational numbers.
Then $\mathbb{R}$ is an infinite dimensional vector space over $\mathbb{Q}$. To prove the main result, we need the following lemmas from \cite{HouLuo}.

\begin{lemma}\label{1}
$\mathbb{R}/\mathbb{Q}$ is an uncountable infinite set.
\end{lemma}

\begin{lemma}\label{2}
For any open interval $(a,b)$, there exists an uncountable infinite subset $A$ of $(a,b)$ such that, for any distinct $x,y\in A$, $x-y$ is an irrational number. Furthermore, for any open interval $(c,d)\subseteq(0,1)$, there exists an uncountable infinite subset $B$ of $(c,d)$ such that, for any distinct $x,y\in B$,
$\ln\frac{x}{1-x}-\ln\frac{y}{1-y}$ is an irrational number.
\end{lemma}

Now let us consider the main conclusion of this article.

\begin{theorem}\label{fnchaos}
There exist a homeomorphism $f$ on the unit closed disk $\overline{\mathbb{D}}\triangleq\{z\in\mathbb{C};|z|\leq 1\}$ such that $f$ is Li-Yorke chaotic but $f^{-1}$ is not.
\end{theorem}
\begin{proof}
Define a mapping $\xi:\mathbb{R}^{+}\rightarrow\mathbb{R}$ by,
$$
\xi(r)=\left\{
\begin{array}{l}
0,\qquad\qquad\, \hbox{if} \ 0\leq r\leq 1
\\[0.25 cm]
\frac{\ln r-[\ln r]}{2^{[\log_{2}([\frac{\ln r}{2}]+1)]}},\qquad\qquad\, \hbox{if} \ r\geq 1 \  \hbox{and} \ [\ln r] \ \hbox{is even}
\\[0.25 cm]
\frac{1-(\ln r-[\ln r])}{2^{[\log_{2}([\frac{\ln r}{2}]+1)]}},\qquad\qquad\, \hbox{if} \ r\geq 1 \  \hbox{and} \ [\ln r] \ \hbox{is odd}
\end{array}
\right.
$$
where $[\cdot]$ means the integer part of some certain real number. Notice that if $\ln r_0=2m$,
$$
\lim\limits_{r\rightarrow r_0^-} \xi(r)=\lim\limits_{r\rightarrow r_0^-} \frac{1-(\ln r-[\ln r])}{2^{[\log_{2}m]}}=0=\xi(r_0),
$$
$$
\lim\limits_{r\rightarrow r_0^+} \xi(r)=\lim\limits_{r\rightarrow r_0^+} \frac{\ln r-[\ln r]}{2^{[\log_{2}(m+1)]}}=0=\xi(r_0);
$$
if $\ln r_0=2m+1$,
$$
\lim\limits_{r\rightarrow r_0^-} \xi(r)=\lim\limits_{r\rightarrow r_0^-} \frac{\ln r-[\ln r]}{2^{[\log_{2}m]}}=\frac{1}{2^{[\log_{2}m]}}=\xi(r_0),
$$
$$
\lim\limits_{r\rightarrow r_0^+} \xi(r)=\lim\limits_{r\rightarrow r_0^+} \frac{1-(\ln r-[\ln r])}{2^{[\log_{2}m]}}=\frac{1}{2^{[\log_{2}m]}}=\xi(r_0)
$$
Then $\xi$ is continuous.

Define a mapping $g:\mathbb{C}\rightarrow\mathbb{C}$ by
$$
g(0)=0 \ \ and \ \ g(w)={e^2w}e^{2\pi i \xi(|w|)}, \ \ \ for \ all \ 0\neq w\in \mathbb{C}.
$$
Then $g$ is a homeomorphism on $\mathbb{C}$ and its inverse is
$$
g^{-1}(0)=0 \ \ and \ \ g^{-1}(w)={e^{-2}w}e^{-2\pi i \xi(e^{-2}|w|)}, \ \ \ for \ all \ 0\neq w\in \mathbb{C}.
$$

Define a mapping $h:\mathbb{D}\rightarrow\mathbb{C}$ by
$$
h(z)=\frac{z}{1-|z|}, \ \ \ for \ all \ z\in \mathbb{D}.
$$
Then $h$ is a homeomorphism and its inverse is
$$
h^{-1}(w)=\frac{w}{1+|w|}, \ \ \ for \ all \ w\in \mathbb{C}.
$$

Furthermore, define a mapping $f:\overline{\mathbb{D}}\rightarrow\overline{\mathbb{D}}$ by
$$
f(z)=\left\{
\begin{array}{l}
\frac{e^2 z}{e^2|z|-|z|+1}e^{2\pi i \xi(\frac{|z|}{1-|z|})}, \qquad \hbox{if} \  0<|z|<1,
\\[0.25 cm]
z,\qquad \hbox{if} \ |z|=0 \ \text{or} \ 1.
\end{array}
\right.
$$
Moreover, for any $|z|<1$, $f(z)=h^{-1}\circ g\circ h(z)$.

Notice that for any $ w\in \mathbb{C}$,
$$
\lim\limits_{n\rightarrow+\infty}|g^{-n}(w)|=0.
$$
Then one can see
$$
\lim\limits_{n\rightarrow+\infty}f^{-n}(z)=\left\{
\begin{array}{l}
0, \qquad \hbox{if}\quad |z|<1,
\\[0.25 cm]
z,\qquad \hbox{if}\quad |z|=1.
\end{array}
\right.
$$
Thus, $f^{-1}$ is not Li-Yorke chaotic.

Now it suffices to show $f$ is Li-Yorke chaotic. Choose an uncountable infinite subset $B$ of the open interval $(\frac{1}{2},\frac{e}{1+e})\subseteq (0,1)$ satisfying the conditions in Lemma \ref{2}.
Given any two distinct points $x,y\in B$, then $\ln\frac{x}{1-x}-\ln\frac{y}{1-y}$ is an irrational number. Consequently, there exist two sequences of strictly increasing positive integers
$\{m_k\}$ and $\{n_k\}$ such that
$$
\lim\limits_{k\rightarrow+\infty}e^{2\pi i m_k(\ln\frac{x}{1-x}-\ln\frac{y}{1-y})}=1 \ \ and \ \
\lim\limits_{k\rightarrow+\infty}e^{2\pi i n_k(\ln\frac{x}{1-x}-\ln\frac{y}{1-y})}=-1.
$$
Notice that given $r\in (1,e)$, for $2^j\leq k \leq 2^{j+1}-1$,
$$
[\log_{2}([\frac{\ln (e^{2k}r)}{2}]+1)]=j.
$$
Then for any $n\in \mathbb{N}$ and $w\in (1,e)$,
\begin{eqnarray*}
g^{2^n-1}(w)&=&e^{2(2^n-1)}we^{2\pi i \sum\limits_{k=0}^{2^n-1}\xi(|e^{2k}w|)} \\
&=&e^{2(2^n-1)}we^{2\pi i \sum\limits_{j=0}^{n-1}\sum\limits_{k=2^j}^{2^{j+1}-1}\xi(|e^{2k}w|)} \\
&=&e^{2(2^n-1)}we^{2\pi i \sum\limits_{j=0}^{n-1}2^j\frac{\ln{w}}{2^j}}\\
&=&e^{2(2^n-1)}we^{2\pi i n\ln w}.
\end{eqnarray*}
Furthermore, for any $z\in(\frac{1}{2},\frac{e}{1+e})$,
$$
f^{2^n-1}(z)=\frac{e^{2(2^n-1)}z}{e^{2(2^n-1)}|z|-|z|+1}e^{2\pi i n\ln \frac{|z|}{1-|z|}}.
$$

Together with
$$
\lim\limits_{n\rightarrow+\infty}\frac{e^{2n}|z|}{e^{2n}|z|-|z|+1}=1 \  \ for \ all \ 0\neq z\in \mathbb{D},
$$
we have
\begin{eqnarray*}
&& \liminf\limits_{n\rightarrow +\infty}|f^{n}(x)-f^{n}(y)|  \\
&\leq& \lim\limits_{k\rightarrow +\infty}|f^{2^{m_k}-1}(x)-f^{2^{m_k}-1}(y)|  \\
&=& \lim\limits_{k\rightarrow+\infty}|\frac{e^{2(2^{m_k}-1)}x}{e^{2(2^{m_k}-1)}x-x+1}e^{2\pi i {m_k}\ln\frac{x}{1-x}}-\frac{e^{2(2^{m_k}-1)}y}{e^{2(2^{m_k}-1)}y-y+1}e^{2\pi i {m_k}\ln\frac{y}{1-y}}|  \\
&\leq& \lim\limits_{k\rightarrow+\infty}|\frac{e^{2(2^{m_k}-1)}x}{e^{2(2^{m_k}-1)}x-x+1}e^{2\pi i {m_k}\ln\frac{x}{1-x}}-\frac{e^{2(2^{m_k}-1)}y}{e^{2(2^{m_k}-1)}y-y+1}e^{2\pi i {m_k}\ln\frac{x}{1-x}}|+  \\
&& \lim\limits_{k\rightarrow +\infty}|\frac{e^{2(2^{m_k}-1)}y}{e^{2(2^{m_k}-1)}y-y+1}e^{2\pi i {m_k}\ln\frac{x}{1-x}}-\frac{e^{2(2^{m_k}-1)}y}{e^{2(2^{m_k}-1)}y-y+1}e^{2\pi i {m_k}\ln\frac{y}{1-y}}|  \\
&=& \lim\limits_{k\rightarrow+\infty}|\frac{e^{2(2^{m_k}-1)}x}{e^{2(2^{m_k}-1)}x-x+1}-\frac{e^{2(2^{m_k}-1)}y}{e^{2(2^{m_k}-1)}y-y+1}||e^{2\pi i {m_k}\ln\frac{x}{1-x}}|+   \\
&&\lim\limits_{k\rightarrow +\infty}|\frac{e^{2(2^{m_k}-1)}y}{e^{2(2^{m_k}-1)}y-y+1}||e^{2\pi i {m_k}\ln\frac{y}{1-y}}||e^{2\pi i {m_k}(\ln\frac{x}{1-x}-\ln\frac{y}{1-y})}-1|   \\
&=& 0.
\end{eqnarray*}
and
\begin{eqnarray*}
&& \limsup\limits_{n\rightarrow +\infty}|f^{n}(x)-f^{n}(y)|  \\
&\geq& \lim\limits_{k\rightarrow +\infty}|f^{2^{n_k}-1}(x)-f^{2^{n_k}-1}(y)|  \\
&=& \lim\limits_{k\rightarrow+\infty}|\frac{e^{2(2^{n_k}-1)}x}{e^{2(2^{n_k}-1)}x-x+1}e^{2\pi i {n_k}\ln\frac{x}{1-x}}-\frac{e^{2(2^{n_k}-1)}y}{e^{2(2^{n_k}-1)}y-y+1}e^{2\pi i {n_k}\ln\frac{y}{1-y}}|  \\
&\geq& \lim\limits_{k\rightarrow+\infty}|\frac{e^{2(2^{n_k}-1)}y}{e^{2(2^{n_k}-1)}y-y+1}e^{2\pi i {n_k}\ln\frac{x}{1-x}}-\frac{e^{2(2^{n_k}-1)}y}{e^{2(2^{n_k}-1)}y-y+1}e^{2\pi i {n_k}\ln\frac{y}{1-y}}|-  \\
&& \lim\limits_{k\rightarrow +\infty}|\frac{e^{2(2^{n_k}-1)}x}{e^{2(2^{n_k}-1)}x-x+1}e^{2\pi i {n_k}\ln\frac{x}{1-x}}-\frac{e^{2(2^{n_k}-1)}y}{e^{2(2^{n_k}-1)}y-y+1}e^{2\pi i {n_k}\ln\frac{x}{1-x}}|  \\
&=& \lim\limits_{k\rightarrow +\infty}|\frac{e^{2(2^{n_k}-1)}y}{e^{2(2^{n_k}-1)}y-y+1}||e^{2\pi i {n_k}\ln\frac{y}{1-y}}||e^{2\pi i {n_k}(\ln\frac{x}{1-x}-\ln\frac{y}{1-y})}-1|-   \\
&&\lim\limits_{k\rightarrow+\infty}|\frac{e^{2(2^{n_k}-1)}x}{e^{2(2^{n_k}-1)}x-x+1}-\frac{e^{2(2^{n_k}-1)}y}{e^{2(2^{n_k}-1)}y-y+1}||e^{2\pi i {n_k}\ln\frac{x}{1-x}}|  \\
&=& 2
\end{eqnarray*}
Thus, $\{x,y\}$ is a Li-Yorke chaotic pair and hence $f$ is Li-Yorke chaotic.
\end{proof}

\end{document}